\documentclass[12pt]{amsart}
\usepackage{amssymb,amsmath,amsthm,amsfonts,amsopn,url,bbm,enumerate}
\copyrightinfo{2009}{American Mathematical Society}
\theoremstyle{plain}
\newtheorem{theorem}{Theorem}[section]
\newtheorem{prop}[theorem]{Proposition}
\newtheorem{lemma}[theorem]{Lemma}

\theoremstyle{definition}

\newtheorem*{definition*}{Definition}
\newtheorem{question*}{Question}

\newtheorem*{example*}{Example}

\theoremstyle{remark}

\newtheorem*{remark*}{Remark}

\numberwithin{equation}{section}

\newcommand{\field}[1]{\mathbbm{#1}}
\newcommand{\N}{\field{N}}

\newcommand{\A}{\field{A}}
\newcommand{\ideal}[1]{\mathfrak{#1}}
\newcommand{\m}{\ideal{m}}

\newcommand{\p}{\ideal{p}}

\newcommand{\ia}{\ideal{a}}
\newcommand{\ib}{\ideal{b}}

\newcommand{\func}[1]{\mathrm{#1} \,}
\newcommand{\rank}{\func{rank}}

\newcommand{\hgt}{\func{ht}}

\newcommand{\ra}{\rightarrow}

\newcommand{\FF}{\mathcal{F}}
\newcommand{\RR}{\mathcal{R}}
\newcommand{\BB}{\mathcal{B}}
\newcommand{\MM}{\mathcal{M}}
\newcommand{\vv}{\mathbf{v}}
\newcommand{\uu}{\mathbf{u}}
\newcommand{\ee}{\mathbf{e}}

\author{Joseph P. Brennan}
  \address{Department of Mathematics \\ University of Central Florida \\4000 Central Florida Blvd., Orlando, FL  32816}
  \email{jpbrenna@mail.ucf.edu}

\author{Neil Epstein}
  \address{Department of Mathematics \\ University of Michigan \\ 530 Church St., Ann Arbor, MI  48109}
  \curraddr{Universit\"at Osnabr\"uck \\ Institut f\"ur Mathematik \\ 49069 Osnabr\"uck \\ Germany}
  \email{nepstein@uos.de}

\title[Noether normalization and matroids]{Noether normalizations, reductions of ideals, and matroids}


\date{\today}

\commby{Bernd Ulrich}

\begin{document}
\begin{abstract}
We show that given a finitely generated standard-graded algebra of dimension $d$ over an infinite field, its graded Noether normalizations obey a certain kind of `generic exchange', allowing one to pass between any two of them in at most $d$ steps.  We prove analogous generic exchange theorems for minimal reductions of an ideal, minimal complete reductions of a set of ideals, and minimal complete reductions of multigraded $k$-algebras.  Finally, we unify all these results into a common axiomatic framework by introducing a new topological-combinatorial structure we call a \emph{generic matroid}, which is a common generalization of a topological space and a matroid.
\end{abstract}

\maketitle

\section{Introduction}
Matroids are a basic combinatorial construction that unites aspects of graph theory, linear algebra over finite fields, and other notions.  By definition, a matroid (with finite basis condition) consists of a ground set $E$ along with a nonempty collection $\BB$ of finite subsets of $E$ (the \emph{bases} of the matroid) such that no element of $\BB$ contains another, and such that for all $B, B' \in \BB$ and $b \in B$, there exists $b' \in B'$ such that $(B \setminus \{b\}) \cup \{b'\} \in \BB$.

Examples of matroids include: field extensions of finite transcendence degree (where the larger field is the ground set), along with their transcendence bases; finite-dimensional vector spaces over a field along with their vector-space bases; the set of edges of a graph with finite diameter along with its spanning forests; simplicial complexes with a particularly strong purity condition.  There are many books on matroids (e.g. \cite{Kung-mat, Oxley-mat, White-tmat}).  Applications of matroid theory to commutative algebra usually use matroids to \emph{create} algebraic structures (e.g. \cite{HerHi-pmat, Sim-clean, Wag-99}).  However, matroids may also be found \emph{within} familiar structures of commutative algebra (c.f. the second named author's \cite{nme-sp, nme*spread}), even when the authors do not explicitly note them (c.f. the obstruction to normality of graph rings \cite{OH} and \cite{SVV} in light of \cite{bicyclic}, or the compatibility property for primary decomposition in \cite{Yao-pdec}).

One of the basic notions of commutative algebra is that of the closure operation on ideals.  Given a ring $R$, a closure operation on $R$-ideals consists of a unary operator $(-)^{c}$ on the set of ideals of $R$ such that for all ideals $J \subseteq I$, we have $J \subseteq (J^{c})^{c} = J^c \subseteq I^c$.  Standard examples include radical, integral closure, and (for rings of prime characteristic) tight closure.  A \emph{c-reduction} \cite{nme*spread} of an ideal $I$ is an ideal $J$ such that $J \subseteq I \subseteq J^c$.  If $J$ is minimal with respect to being a $c$-reduction of $I$, we call it a \emph{minimal c-reduction} of $I$.  If $c$ is integral closure, we drop the closure symbol and use the terms \emph{reduction} and \emph{minimal reduction}, following the terminology in \cite{NR}.

Northcott and Rees \cite{NR} showed that for an ideal $I$ in a Noetherian local ring $R$, $I$ has minimal reductions, and that if $R$ has infinite residue field then the minimal reductions of $I$ all have the same cardinality.  On the other hand, one of the characteristic qualities of a matroid is that all bases have the same cardinality (the \emph{rank} of the matroid).
Accordingly a natural question to ask is whether the minimal generating sets of the minimal reductions of  such an ideal $I$ form the bases of a matroid.
The following counterexample shows that the answer is ``no'', even in the case where $I=\m$ and $R$ is a quadric hypersurface:

\begin{example*}
Let $R=k[\![x,y,z,w]\!]/(xy-zw)$, and let $I=\m = (x,y,z,w)$.  Consider the sets $A := \{x+y,z,w\}$ and $B := \{x, y, z+w\}$.  The ideals $\ia := (A)$ and $\ib := (B)$ are both minimal reductions of $\m$.  However, if one omits the element $x+y$ from $A$ and tries to replace it with an element from $B$, one notices that none of the sets $\{x,z,w\}$, $\{y, z, w\}$, $\{z+w, z,w\}$ generate reductions of $\m$.  Indeed, all generate prime ideals which are properly contained in $\m$.
\end{example*}

However, the answer is ``yes"  \cite{nme-sp, nme*spread} if one replaces ``minimal reduction" with ``minimal $c$-reduction", where $c$ is any of: Frobenius closure, tight closure, or plus closure, under mild conditions on a ring of prime characteristic.

In this note, we provide a more general, topological-combinatorial structure (which we call a \emph{generic matroid}) that includes all matroids as a special case, but also explains minimal reductions of ideals, graded Noether normalizations of standard graded algebras, and certain ``complete reductions'' (c.f. Rees \cite{Rees-mix} or O'Carroll \cite{Oca-87}).  The idea (in the case of minimal reductions of an ideal, say) is that given two minimal reductions, if one removes one of the generators from the first reduction, ``almost all'' choices of a minimal generator from the second reduction will work to complement the generating set to give a third minimal reduction.

After first providing the aforementioned examples (in Theorems~\ref{thm:algebras}, \ref{thm:reductions}, \ref{thm:mulalgebras}, \ref{thm:mulalgebras2}, and \ref{thm:completereds}) of matroid-like exchange, we give the definition of a generic matroid in Section~\ref{sec:gmat}, capturing all the examples into a single axiomatic framework. For general background on integral closure, minimal reductions, and related matters, we recommend the book \cite{HuSw-book}.

\section{Exchange in graded Noether normalizations}

Our first example of ``generic exchange'' happens in the context of standard-graded $k$-algebras, where $k$ is an infinite field.  To fix notation and terms, we provide the following very familiar definition:
\begin{definition*}
Let $k$ be a field.  A $k$-algebra $S$ is called \emph{standard graded algebra} if it satisfies the following conditions: \begin{itemize}
\item It is $\N$-graded and finitely generated as a $k$-algebra,
\item $S_0=k$, and
\item $S=k[S_1]$.
\end{itemize}
A ring homomorphism $g: R \ra S$ between $\N$-graded algebras is called \emph{graded} if $g(R_n) \subseteq S_n$ for all $n \in \N$.

If $R$ is a standard graded algebra over $k$ of dimension $d$, then a \emph{(graded) Noether normalization} of $R$ is an injective (graded) ring homomorphism $g: A \hookrightarrow R$, such that $A$ is a polynomial ring in $d$ variables over $k$ (which is standard graded over $k$ in such a way that all variables are homogeneous of degree 1), such that $R$ is finitely generated as an $A$-module.
\end{definition*}

Recall the following theorem on graded Noether normalization:

\begin{prop}\cite[Theorem 1.5.17]{BH}\label{prop:NN}
Let $k$ be a field and $R$ a positively graded affine $k$-algebra.  Set $n = \dim R$. \begin{enumerate}
\item[\textrm{(a)}] The following are equivalent for homogeneous elements $x_1, \dotsc, x_n$:
 \begin{enumerate}
  \item[\textit{(i)}] $x_1, \dotsc, x_n$ is a homogeneous system of parameters;
  \item[\textrm{(ii)}] $R$ is an integral extension of $k[x_1, \dotsc, x_n]$;
  \item[\textrm{(iii)}] $R$ is a finite $k[x_1, \dotsc, x_n]$-module.
 \end{enumerate}
\item[\textrm{(b)}] There exist homogeneous elements $x_1, \dotsc, x_n$ satisfying the conditions in \textrm{(a)}.  Moreover, such elements are algebraically independent over $k$.
\item[\textrm{(c)}]  If $R$ is a standard graded algebra over $k$ and $k$ is infinite, then such $x_1, \dotsc, x_n$ can be chosen to be of degree 1.
\end{enumerate}
\end{prop}

\begin{theorem}\label{thm:algebras}
Let $S$ be a standard graded $k$-algebra of dimension $d\geq 1$, where $k$ is an infinite field, and let $A = k[X_1, \dotsc, X_d] \subseteq S$ and $B=k[Y_1, \dotsc, Y_d] \subseteq S$ be graded Noether normalizations of $S$.  There is a proper $k$-subspace arrangement $M \subsetneq k^d$ such that for any $c = [c_1, \dotsc, c_d] \in k^d$, we have $c \in k^d \setminus M$ if and only if the inclusion $k[X_1, \dotsc, X_{d-1}, \sum_{j=1}^d c_j Y_j] \subseteq S$ is a graded Noether normalization.
\end{theorem}

\begin{proof}
Let $J := (X_1, \dotsc, X_{d-1})S$.  Let $\p_1, \dotsc, \p_r$ be the minimal primes of $J$.  Note that each $\p_j$ is homogeneous \cite[Lemma 1.5.6]{BH}.  For each $j=1, \dotsc, r$, let $V_j$ be the $k$-vector space consisting of homogeneous degree 1 elements of $\p_j$.  Let $V := \{\sum_{j=1}^d c_j Y_j \mid [c_1, \dotsc, c_d] \in k^d \}$.  Since the $Y_j$ are algebraically independent over $k$, they are linearly independent as well, so $V$ is $k$-isomorphic to $k^d$.  Let \[
N := \{y \in V \mid k[X_1, \dotsc, X_{d-1}, y]\subseteq S \text{ is \emph{not} a Noether normalization}\}.
\]

It follows from Proposition~\ref{prop:NN} and dimension considerations that $N$ has an alternate description, namely: \[
N = \{y \in V \mid \{X_1, \dotsc, X_{d-1}, y\} \text{ is not a homogeneous SOP for } S\}.
\]

So for $y \in V$, we have \[
y\in N \iff y \in \bigcup_i \p_i \iff y \in \bigcup_{1 \leq i \leq r} V_i,
\]
because $y$ is homogeneous of degree 1.

That is, $N = \left(\bigcup_{i=1}^r V_i\right) \cap V = \bigcup_{i=1}^r (V_i \cap V)$

On the other hand, we have $L := (Y_1, \dotsc, Y_d)S \nsubseteq \p_j$, since $L$ is $S_+$-primary and $\p_j$ is a prime properly contained in $S_+$.  Since $L$ is generated by $V$,  we have $V \nsubseteq \p_j$, and hence $V \nsubseteq V_j$ since $V_j \subseteq \p_j$.  Then by ``vector space avoidance'', we have $V \nsubseteq \bigcup_j V_j$.  But $N \subseteq \bigcup_j V_j$, so it follows that $N \subsetneq V$.

The subset $M$ of $k^d$ corresponding to $N$ has the required property.
\end{proof}

A more general version of vector space avoidance can be found at the end of the proof of Theorem~\ref{thm:mulalgebras}.

\begin{remark*}
One may tweak the proof above to get a more general result.  Namely, start by assuming that $S$ is finitely generated and positively-graded over $k$, but not standard-graded.  Then let $X_i$ and $Y_i$ be homogeneous (but not necessarily degree 1) elements of $S$ in such a way that $A$ and $B$ are Noether normalizations of $S$.  Then $k[X_1, \dotsc, X_{d-1}, \sum_{j=1}^d c_j Y_j]$ is a Noether normalization of $S$ if and only if $c\in k^d$ avoids a proper subspace arrangement $M$.

The tweak in the proof goes as follows: Let $\delta := \max \{\deg Y_i \mid 1\leq i \leq d\}$.  Let $\p_1, \dotsc, \p_r$ be as above, and for each $1\leq j\leq r$, let $V_j$ be the $k$-vector space spanned by the homogeneous elements of $\p_j$ of positive degrees $\leq \delta$.  Then the rest of the proof goes through without change.
\end{remark*}

\section{Exchange in minimal reductions of ideals}\label{sec:minred}

As a basic reference for the background material for this section (including all statements in the following paragraph), we recommend \cite[Chapters 5 and 8]{HuSw-book}.

Let $(R,\m,k)$ be a local ring.  For an ideal $I$, we denote the Rees ring of $I$ by $\RR(I) := R[It] = \bigoplus_{n \geq 0} I^n t^n \subseteq R[t]$, where $I^0 := R$ and $t$ is indeterminate.  We denote the fiber ring of $I$ by $\FF(I) := \RR(I) \otimes_R k = \bigoplus_{n\geq 0} (I^n / \m I^n) t^n$.  Recall that $\ia$ is a reduction of $I$ iff $\RR(I)$ is module-finite over $\RR(\ia)$, iff $\FF(I)$ is module-finite over the image of $\FF(\ia)$.  Also, a reduction $\ia$ of $I$ is a minimal reduction iff $\FF(\ia)$ is a polynomial ring over $k$, iff $\RR(\ia)$ is a polynomial ring over $R$.  In this case, we have that the natural map $\FF(\ia) \ra \FF(I)$ is injective (which is \emph{not} true in general if $\ia$ is a non-minimal reduction of $I$).  The \emph{analytic spread} of $I$ is defined to be $\ell(I) := \dim \FF(I)$.  When $R$ is Noetherian and $k$ is infinite, $\ell(I)$ coincides with the minimal number of generators of any minimal reduction of $I$.

\begin{theorem}\label{thm:reductions}
Let $(R,\m,k)$ be a Noetherian local ring with infinite residue field.  Let $I$ be a proper ideal of $R$ with analytic spread $d>0$, and let $\ia$, $\ib$ be minimal reductions of $I$.  Let $a_1, \dotsc, a_d$ be a minimal generating set of $\ia$.  Then for \emph{almost all} minimal generators $b$ of $\ib$, the ideal $(a_1, \dotsc, a_{d-1}, b)$ is a minimal reduction of $I$.  In particular, there is a finite set of ideals $J_i$ such that $\m \ib \subseteq J_i$ for each $i$ and $\bigcup_i J_i \subsetneq \ib$, such that $b\in \ib$ works if and only if $b \notin \bigcup_i J_i$.
\end{theorem}

\begin{proof}
Fix a minimal generating set $b_1, \dotsc, b_d$ of $\ib$ as well.

Let $S=\FF(I)$, $A = \FF(\ia)$, and $B = \FF(\ib)$, and consider $A$ and $B$ to be subrings of $S$ via the injections outlined above.  Let $X_1, \dotsc, X_d$ be the degree-1 images in $A$ of $a_1, \dotsc, a_d$ respectively (i.e. $X_i = \overline{a_i} t$ for each $i$), and let $Y_1, \dotsc, Y_d$ be the degree-1 images in $B$ of $b_1, \dotsc, b_d$ respectively.  The current data now match the hypotheses of Theorem~\ref{thm:algebras}.  Hence, there is a proper $k$-subspace arrangement $M \subsetneq k^d$ such that for any $c=[c_1, \dotsc, c_d] \in k^d$, $C=k[X_1, \dotsc, X_{d-1}, \sum_{j=1}^d c_j Y_j] \subseteq S$ is a graded Noether normalization iff $c \notin M$.  So for any $r=[r_1, \dotsc, r_d] \in R^d$, letting $\bar{r} := c \in k^d$, $(a_1, \dotsc, a_{d-1}, \sum_{j=1}^d r_j b_j)$ is a minimal reduction of $I$ iff $c \notin M$.  But the preimage of $M$ in $R^d$, dotted with the vector $\langle b_1, \dotsc, b_d\rangle$, is a finite union $H$ of ideals $J_i$ with $\m \ib \subseteq J_i$ and $H \subsetneq \ib$, as claimed.
\end{proof}

\section{Complete reductions, and some notes on terminology}
The notion of a \emph{complete reduction} of a set of ideals was first defined by Rees \cite{Rees-mix}, and later presented in a more general context by O'Carroll \cite{Oca-87, Oca-05}.  Note also that complete reductions give rise to \emph{joint reductions},
which are connected to the notion of mixed multiplicities and of great interest in their own right \cite{MR1483767,MR1467469,Rees-mix,MR959271,MR1154678,MR1302856,MR1726287,MR1085400}

\begin{definition*}[O'Carroll]
Let $(A,\m,k)$ be a Noetherian local ring, and let $I_1, \dotsc, I_n$ be a sequence of ideals of $A$.  A \emph{complete reduction} of $I_1, \dotsc, I_n$ \emph{of type $r$} is an $n\times r$ matrix $\{a_{ij}\}$ such that $a_{ij} \in I_i$ for each $i, j$, and such that if we set $I := \prod_{i=1}^n I_i$ and $b_j := \prod_{i=1}^n a_{ij}$ for each $1\leq j\leq r$, the ideal $(b_1, \dotsc, b_r)$ is a reduction of $I$.
\end{definition*}

He showed that complete reductions of type $r$ exist iff $r\geq \ell(I)$.

Now we fix some conventions for multigraded $k$-algebras:  Fix a field $k$.  For an $\N^n$-graded $k$-algebra $S$, for $1\leq i \leq n$ we denote $S^{(i)} := S_{(0,\dotsc, 0, 1, 0, \dotsc, 0)}$, where the 1 is in the $i$th spot.  Such a $k$-algebra is called \emph{standard} if $S_{(0, \dotsc, 0)} = k$ and $S$ is generated as a $k$-algebra by $S^{(1)} \cup \cdots \cup S^{(n)}$.  We define the \emph{diagonal subring} $S^\Delta := k[S_{(1, 1, \dotsc, 1)}]$, with $\N$-grading given by setting the degree of any element of $S_{(n, n, ..., n)}$ to $n$.

If $I_1, \dotsc, I_n$ are ideals of a local ring $(A,\m, k)$, the \emph{(multigraded) fiber ring} $S = \FF(I_1, \dotsc, I_n)$ is the standard $\N^n$-graded $k$-algebra defined by: \[
S_{(r_1, \dotsc, r_n)} := \frac{I_1^{r_1} \cdots I_n^{r_n}}{\m I_1^{r_1} \cdots I_n^{r_n}} t_1^{r_1} \cdots t_n^{r_n},
\]
with multiplication induced from that of $A$, where $t_1, \dotsc, t_n$ are indeterminates over $k$.

As we have seen, it is useful to translate ideal-theoretic ideas into ring-theoretic terms, and such a translation was undertaken by Kirby and Rees \cite{KR-mul1}, where the notion of the complete reduction of a multigraded ring is defined.  However, their notion corresponds to a different generalization of Rees' original ideal-theoretic notion than O'Carroll's generalization in \cite{Oca-87} and \cite{Oca-05}.  Since we find O'Carroll's ideal-theoretic notion more convenient than that of Kirby and Rees, we make the following ring-theoretic definition.

\begin{definition*}
Let $S$ be a standard $\N^n$-graded $k$-algebra.  A \emph{complete reduction} of $S$ \emph{of type $r$} is an $n \times r$ matrix $\{x_{ij}\}$ where each $x_{ij} \in S^{(i)}$, such that if for each $1\leq j \leq r$ we set $X_j := \prod_{i=1}^n x_{ij}$, $S^\Delta$ is module-finite over the subring $k[X_1, \dotsc, X_r]$.
\end{definition*}

\begin{lemma}\label{lem:multigraded}
Let $(A,\m, k)$ be a Noetherian local ring, and let $I_1, \dotsc, I_n$ be proper ideals of $A$.  Let $S := \FF(I_1, \dotsc, I_n)$ be the multigraded fiber ring of $I_1, \dotsc, I_n$.  Let $r \in \N$ and let $\{a_{ij}\}$ be an $n \times r$ matrix such that $a_{ij} \in I_i$ for each $i,j$.  For each $i,j$, let $x_{ij}$ be the image of $a_{ij}$ in $S^{(i)}$.  Then $\{a_{ij}\}$ is a complete reduction of the ideals $I_1, \dotsc, I_n$ if and only if $\{x_{ij}\}$ is a complete reduction of $S$.
\end{lemma}

\begin{proof}
Let $I := \prod_{i=1}^n I_i$ and $T := S^\Delta= \FF(I)$.  It is enough to prove the more general statement that for any sequence $b_1, \dotsc, b_r$ of elements of $I$, if $y_1, \dotsc, y_r$ are the images of the $b_j$ in $T_1 = I/\m I$, then $(b_1, \dotsc, b_r)$ is a reduction of $I$ if and only if $T$ is module-finite over the subring $k[y_1, \dotsc, y_r]$.  But this was covered in the introduction to Section~\ref{sec:minred}.
\end{proof}

\section{Exchange in complete reductions of multigraded $k$-algebras}

Let $S$ be a standard $\N^n$-graded $k$-algebra, and $d := \dim S^\Delta$.  If there is a complete reduction of type $r$, then $S^\Delta$ is module-finite over an $r$-generated $k$-algebra $B$, hence $d = \dim S^\Delta = \dim B \leq r$.  Moreover, if equality holds, then there can be no relations among the $d$ generators of $B$ over $k$, whence $B$ is a polynomial ring over $k$, so that $B \subseteq S^\Delta$ is a graded Noether normalization of $S^\Delta$.    It is natural to call a complete reduction \emph{minimal} if it is of type $d$.

\begin{theorem}\label{thm:mulalgebras}
Let $S$ be a standard $\N^n$-graded $k$-algebra, where $k$ is an infinite field, and let $S^\Delta$ be its diagonal subring.  Suppose that $d := \dim S^\Delta \geq 1$, and that $\{x_{ij} \mid 1\leq i \leq n\text{, }1\leq j \leq d\}$ and $\{w_{ij} \mid 1\leq i \leq n\text{, }1\leq j \leq d\}$ are complete reductions of $S$.  Then there is a proper Zariski-closed subset $M \subsetneq k^{n \times d}$ such that for any $C = \{c_{ij}\} \in k^{n \times d}$, we have $C \notin M$ if and only if $\{x^C_{ij}\}$ is a complete reduction of $S$, where \[
x^C_{ij} := \begin{cases}
x_{ij} &\text{ if } j < d, \\
\sum_{h=1}^d c_{ih} w_{ih} &\text{ if } j=d. \end{cases}
\]
\end{theorem}

\begin{proof}
Let $w^C_i := \sum_{j=1}^d c_{ij} w_{ij}$, and \begin{align*}
w^C &:= \prod_{i=1}^n w_i^C = \prod_{i=1}^n \left(\sum_{j=1}^d c_{ij} w_{ij} \right) \\
&= \sum_{(j_1,\dotsc, j_n) \in \{1, \dotsc, d\}^n} c_{1 j_1} c_{2 j_2} \cdots c_{n j_n} w_{1 j_1} \cdots w_{n j_n}.
\end{align*}
Let $u := \dim_k (S^\Delta)_1$.  Then after the identification $(S^\Delta)_1 \cong \A^u_k$, we have a morphism $g: \A^{nd}_k \ra \A^u_k$ (of degree $n$), given by $g(C) := w^C$.

For each $1 \leq j \leq d$, let $X_j := \displaystyle \prod_{i=1}^n x_{ij}$ and $W_j := \displaystyle \prod_{i=1}^n w_{ij}$.  Let $\p_1, \dotsc, \p_r$ be the minimal primes in $S^\Delta$ of the ideal $J := (X_1, \dotsc, X_{d-1}) \subseteq S^\Delta$.  Since $J$ is homogeneous, each $\p_i$ is homogeneous.  For $1\leq i \leq r$, let $V_i := \p_i \cap (S^\Delta)_1$.  Then each $V_i$ is a $k$-linear subspace of $k^u$, hence a closed subvariety of $(S^\Delta)_1=\A^u_k$.

Now let $M := \{C \in k^{n\times d} \mid \{x_{ij}^C\} \text{ is not a complete reduction of } S\}$.  We have that $C \in M \iff k[X_1, \dotsc, X_{d-1}, w^C] \subseteq S^\Delta$ is not a Noether normalization $\iff w^C \in \bigcup_i \p_i \iff g(C) \in \bigcup_i V_i \iff C \in \bigcup_i g^{-1}(V_i)$.  But for each $i$, $g^{-1}(V_i)$ is a closed subvariety of $\A^{nd}_k$, since $V_i$ is a closed subvariety of $\A^u_k$.

Suppose for some $i$ that $g^{-1}(V_i) = \A^{nd}_k$.  For each $1\leq k \leq d$, let $B^{(j)}$ be the $n\times d$ matrix where all entries in the $j$th column are $1$s, and all other entries are $0$.  We have for each $j$ that $B^{(j)} \in g^{-1}(V_i)$, so that (using the notation from the beginning of the proof): \[
W_j = \prod_{i=1}^n w_{ij} = \prod_{i=1}^n w_i^{B^{(j)}} = w^{B^{(j)}} = g(B^{(j)}) \in V_i.
\]
Thus, the $k$-span of the $W_j$ is contained in $V_i$, hence in $\p_i$, so that 
$(W_1, \dotsc, W_d) \subseteq \p_i$.   On the other hand, since $\{w_{ij}\}$ is a complete reduction of $S$, it follows that $S^\Delta$ is module-finite over $k[W_1, \dotsc, W_d]$, whence $\hgt (W_1, \dotsc, W_d) = \dim S^\Delta = d$.  This puts a height $d$ ideal $(W_1, \dotsc, W_d)$ inside a height $d-1$ ideal $\p_i$, yielding a contradiction.

Hence, for each $1 \leq i \leq r$, $g^{-1}(V_i)$ is a \emph{proper} subvariety of $\A^{nd}_k$.  Since a finite union of proper subvarieties of affine space over an infinite field must still be proper (i.e. $\A^{nd}_k$ is irreducible; see for instance \cite[Chapter II: Rule 1.2 and Proposition 3.11]{Kunz-book}), it follows that $M = \bigcup_i g^{-1}(V_i)$ is a proper subvariety of $\A^{nd}_k$, which implies that topologically it is a proper Zariski-closed subset of $k^{nd}$.
\end{proof}

A slight variant of the above theorem also holds, namely:

\begin{theorem}\label{thm:mulalgebras2}
Let $S$ be a standard $\N^n$-graded $k$-algebra, where $k$ is an infinite field, and let $S^\Delta$ be its diagonal subring.  Suppose that $d := \dim S^\Delta \geq 1$, and that $\{x_{ij} \mid 1\leq i \leq n\text{, }1\leq j \leq d\}$ and $\{w_{ij} \mid 1\leq i \leq n\text{, }1\leq j \leq d\}$ are complete reductions of $S$.  Then there is a proper Zariski-closed subset $M \subsetneq k^d$ such that for any $\vv = \{v_j\} \in k^d$, we have $\vv \notin M$ if and only if $\{x^{\vv}_{ij}\}$ is a complete reduction of $S$, where \[
x^\vv_{ij} = \begin{cases}
x_{ij} &\text{if } j<d,\\
\sum_{h=1}^d v_h w_{ih} &\text{if } j=d.
\end{cases}
\]
\end{theorem}

\begin{proof}
The proof is very similar to that of Theorem~\ref{thm:mulalgebras}, so we use the same notations as in that proof, with the exception of the meaning of the symbol $M$.  In this case, we get a polynomial map $f: \A^d_k \ra \A^u_k$, and if $M$ is the set of all vectors that do not result in a complete reduction of $S$, then $M = \bigcup_i f^{-1}(V_i)$.

All that needs to be shown is that $f^{-1}(V_i) \neq \A^d_k$.  Let $\ee_j$ be the $j$th standard basis vector of $k^d$ for $1 \leq j  \leq d$.  If $f^{-1}(V_i)= \A^d_k$, then $f(\ee_j) \in V_i$ for each $j$.  But $f(\ee_j) = W_j$, and we get the same contradiction as in the proof of the preceding theorem.
\end{proof}

\section{Exchange in complete reductions of sequences of ideals}

According to \cite{Oca-87}, complete reductions of $I_1, \dotsc, I_n$ of type $r$ exist if and only if $r \geq \dim S^\Delta = \ell(\prod_{i=1}^n I_i)$, where $S = \FF(I_1, \dotsc, I_n)$ (so that $S^\Delta = \FF(I)$).

With this background, we transform Theorems~\ref{thm:mulalgebras} and \ref{thm:mulalgebras2} into the following theorem about minimal complete reductions of sequences of ideals, in the same way that we transformed Theorem~\ref{thm:algebras} into Theorem~\ref{thm:reductions}, as follows:

\begin{theorem}\label{thm:completereds}
Let $(R,\m,k)$ be a Noetherian local ring, where $k$ is infinite, and let $I_1, \dotsc, I_n$ be a sequence of ideals of $R$, $I := \prod_{i=1}^n I_i$.  Suppose that $d := \ell(I) \geq 1$, and that $\{a_{ij}\}$ and $\{b_{ij}\}$ are complete reductions of $I_1, \dotsc, I_n$ of type $d$.  Then for \emph{almost all} matrices $Z = \{z_{ij}\} \in R^{n \times d}$ (resp. vectors $\uu = \{u_j\} \in R^d$), we have that $\{a_{ij}^Z\}$ (resp. $\{a_{ij}^\uu\}$) is a complete reduction of $I_1, \dotsc, I_n$, where \[
a_{ij}^Z := \begin{cases}
a_{ij} &\text{if } j<d,\\
\sum_{h=1}^d z_{ih} b_{ih} &\text{if } j=d.
\end{cases}
\] (resp. \[
a_{ij}^\uu := \begin{cases}
a_{ij} &\text{if } j<d,\\
\sum_{h=1}^d u_h b_{ih} &\text{if } j=d.\text{)}
\end{cases}
\]
\end{theorem}

\begin{proof}
Let $S := \FF(I_1, \dotsc, I_n)$ be the multigraded fiber ring, and let $\{x_{ij}\}$ and $\{y_{ij}\}$ be the images in $S^{(i)}$ of $\{a_{ij}\}$, $\{b_{ij}\}$ respectively.

By Lemma~\ref{lem:multigraded}, $\{x_{ij}\}$ and $\{y_{ij}\}$ are complete reductions of $S$, and since $\ell(I) = \dim S^\Delta = d$, they are minimal.  Then by Theorem~\ref{thm:mulalgebras} (resp. Theorem~\ref{thm:mulalgebras2}), there is a proper Zariski-closed subset $M \subsetneq k^{n\times d}$ (resp. $\subsetneq k^d$) such that for matrices $C \in k^{n\times d}$ (resp. vectors $\mathbf{v} \in k^d$), we have $C$ (resp. $\vv$) $\notin M$ iff $\{x_{ij}^C\}$ (resp. $\{x_{ij}^\vv\}$) is a complete reduction of $S$.  By Lemma~\ref{lem:multigraded} again, this holds iff $\{a_{ij}^Z\}$ (resp. $\{a_{ij}^\uu\}$) is a complete reduction of $I_1, \dotsc, I_n$ for every, equivalently some, lifting $Z$ of $C$ (resp. lifting $\uu$ of $\vv$).
\end{proof}

\section{Generic matroids: the unifying context}\label{sec:gmat}

\begin{definition*}
Let $E$ be a set, $\tau$ a topology on $E$, $\BB$ a collection of finite subsets of $E$, and $\MM$ a nonempty collection of subsets of $E$.  For each $M \in \MM$, let $\BB_M := \{B \in \BB \mid B \subseteq M\}$, and $C_M := \bigcup \{B \mid B \in \BB_M\}$.  We say that the tuple $(E, \tau, \BB, \MM)$ is a \emph{generic matroid} if the following conditions hold: \begin{enumerate}
\item For each $M \in \MM$, the set $M$ equipped with the collection $\BB_M$ as bases forms a matroid.
\item $\BB = \bigcup_{M\in \MM} \BB_M$.
\item (Exchange property): For any $b \in B \in \BB$ and $M \in \MM$, there exists a $\tau$-open set $U$ such that $U' := U \cap C_M \neq \emptyset$ and such that for any $x \in M$, $x \in U'$ if and only if $(B \setminus \{b\}) \cup \{x\} \in \BB$.
\end{enumerate}
\end{definition*}

The first thing to note is that if $(E,\tau,\BB,\MM)$ is a generic matroid and $\sigma$ is a finer topology than $\tau$, then $(E,\sigma,\BB,\MM)$ is also a generic matroid.  That is, a coarser topology amounts to a stronger statement.  Next, we note that generic matroids satisfy the same equicardinality condition that matroids do:

\begin{prop}\label{prop:equicard}
Let $(E,\tau, \BB, \MM)$ be a generic matroid.  Then for any $B, B' \in \BB$, $\#(B) = \#(B')$.
\end{prop}

\begin{proof}
Let $B = \{a_1, \dotsc, a_n\} \in \BB$ and $M \in \MM$.  By the exchange property, there exists $x_1 \in C_M$ such that $\{x_1, a_2, \dotsc, a_n\} \in \BB$.  After $n-1$ more iterations of the same property, we find successive elements $x_2, x_3, \dotsc, x_n \in C_M$ to replace $a_2, a_3,  \dotsc, a_n$ respectively, so that $X := \{x_1, x_2, \dotsc, x_n\} \in \BB$.  But since $X \subseteq M$, we have $X \in \BB_M$, so that $n = \#(X) = \rank M$ (by which we mean the rank of the matroid $M$ in the ordinary sense -- \emph{i.e.} the cardinality of any basis of $M$), whence for any $B' \in \BB_M$, we have $\#(B) = n= \#(B')$.  Since $M$ was arbitrary, and $\BB = \bigcup_{M\in \MM} \BB_M$, the proof is complete.
\end{proof}

In fact, the above proof shows more.  It shows that given any pair of matroids $M, M' \in \MM$ of a generic matroid $(E,\tau,\BB,\MM)$, we can get from a given basis of $M$ to some basis of $M'$ in at most $d$ `steps', where $d$ is the common rank of $M$ and $M'$.

We finish with examples of generic matroids:

\begin{example*}[Matroids]
Consider an ordinary matroid with ground set $E$ and bases $\BB$.  Let $\tau$ be the discrete topology on $E$.  Then it is easy to show that both $(E, \tau, \BB, \{E\})$ and $(E, \tau, \BB, \BB)$ are generic matroids.
\end{example*}

\begin{example*}[Topological spaces]
For a set $E$ with a given topology $\tau$, $(E, \tau, \{\emptyset\}, \{E\})$ is a generic matroid.
\end{example*}

\begin{example*}[Graded Noether normalizations]
Let $S$ be a standard $\N$-graded $k$-algebra, and let $k$ be an infinite field.  Let $E := S$ and let $\MM$ be the set of subalgebras $A$ such that $A \subseteq S$ is a graded Noether normalization.  We use the topology $\tau$ where the closed sets consist of all finite unions of vector subspaces of $S_1$.  Let $\BB$ be minimal $k$-algebra generating sets taken from $S_1$ of elements of $\MM$.  If $\dim S=0$, then $\MM=\{k\}$ and $\BB=\{\emptyset\}$, so  $(E,\tau,\BB,\MM)$ is trivially a generic matroid.  On the other hand, if $\dim S>0$, then Theorem~\ref{thm:algebras} shows that $(E,\tau,\BB,\MM)$ is a generic matroid.
\end{example*}

\begin{example*}[Minimal reductions of ideals]
Let $(R,\m, k)$ be a Noetherian local ring with infinite residue field, and $I$ a proper ideal.  Let $E := I$, let $\MM$ be the minimal reductions of $I$, let $\BB$ be the minimal generating sets of elements of $\MM$, and let $\tau$ be the topology whose closed sets consist of finite unions of ideals $J$ such that $\m I \subseteq J \subseteq I$.  If $I$ is nilpotent, then $\MM=\{0\}$ and $\BB=\{\emptyset\}$, so  $(E,\tau,\BB,\MM)$ is trivially a generic matroid.  If $I$ is non-nilpotent, then $\ell(I)>0$, so Theorem~\ref{thm:reductions} shows that $(E,\tau,\BB,\MM)$ is a generic matroid (even though it is not necessarily a matroid in the ordinary sense, as shown in the example from section 1.  Essentially the same example shows that ordinary matroids do not suffice in the case of graded Noether normalizations either.).
\end{example*}

\begin{example*}[Complete reductions of multigraded algebras]
Let $S$ be a standard $\N^n$-graded $k$-algebra, where $k$ is an infinite field and $d := \dim S^\Delta$.  Let $E := S^{(1)} \times \cdots \times S^{(n)}$, with elements considered as column vectors.  Let $\BB$ consist of those $d$-tuples $\{\mathbf{x}_j\}$ of elements of $E$ such that the corresponding matrix $\{x_{ij}\}$ is a complete reduction of $S$.  Let $\MM$ be the collection $k$-subspaces of $E$ spanned by elements of $\BB$.  For $\tau$, we may use the Zariski topology on the $k$-linear space $E$ (although there are better choices).    If $d=0$, then $\BB=\{\emptyset\}$ and $\MM = \{\mathbf{0}\}$, so  $(E, \tau,\BB,\MM)$ is trivially a generic matroid.  If $d>0$, then Theorem~\ref{thm:mulalgebras2} shows that $(E, \tau,\BB,\MM)$ is a generic matroid.
\end{example*}

\begin{example*}[Complete reductions of sequences of ideals]
Let $(R,\m,k)$ be a Noetherian local ring with infinite residue field, let $I_1, \dotsc, I_n$; let $I := \prod_{i=1}^n I_i$ (the ideal-theoretic product, not the set-theoretic one) and $d := \ell(I)$.  Let $E := I_1 \times \cdots \times I_n$ (the set-theoretic product), with elements considered as column vectors.  Let $\BB$ consist of those $d$-tuples $\{\mathbf{a}_j\}$ of elements of $E$ such that the corresponding matrix $\{a_{ij}\}$ is a complete reduction of $I_1, \dotsc, I_n$.    Let $\MM$ be the collection of $R$-submodules of $E$ generated by elements of $\BB$.  Let the closed sets of $\tau$ consist of all preimages of Zariski-closed subsets of $I_1/\m I_1 \times \cdots \times I_n / \m I_n$.  If $d=0$, then $(E,\tau,\BB,\MM)$ is trivially a generic matroid.  If $d>0$, then Theorem~\ref{thm:completereds} shows that $(E,\tau,\BB,\MM)$ is a generic matroid.
\end{example*}

\providecommand{\bysame}{\leavevmode\hbox to3em{\hrulefill}\thinspace}
\providecommand{\MR}{\relax\ifhmode\unskip\space\fi MR }
\providecommand{\MRhref}[2]{%
  \href{http://www.ams.org/mathscinet-getitem?mr=#1}{#2}
}
\providecommand{\href}[2]{#2}

\end{document}